\documentclass[11pt,reqno]{amsart}
\usepackage{enumerate}
\usepackage{amsrefs}
\usepackage{amsfonts, amsmath, amssymb, amscd, amsthm}
\usepackage{url}
\usepackage{graphicx}
\usepackage[
linktocpage=true,colorlinks,citecolor=magenta,linkcolor=blue,urlcolor=magenta]{hyperref}
\usepackage{multicol}
\usepackage{comment}
\usepackage{tikz}
\usepackage[margin=0.6in]{geometry}
\newtheorem{thm}{Theorem}[section]
\newtheorem*{thmA}{Theorem A}
\newtheorem*{thmB}{Theorem B}

\newtheorem{cor}[thm]{Corollary}

\newtheorem{lem}[thm]{Lemma}

\newtheorem{prop}[thm]{Proposition}
\theoremstyle{definition}

\newtheorem*{ack}{Acknowledgments}
\theoremstyle{remark}
\newtheorem{rem}[thm]{Remark}

\numberwithin{equation}{section}
\allowdisplaybreaks

\renewcommand{\(}{\left(}
\renewcommand{\)}{\right)}
\renewcommand{\~}{\tilde}
\renewcommand{\-}{\overline}

\renewcommand{\a}{\alpha}
\renewcommand{\b}{\beta}

\renewcommand{\d}{\delta}
\newcommand{\e}{\varepsilon}
\renewcommand{\k}{\kappa}

\newcommand{\D}{\Delta}

\newcommand{\G}{\Gamma}

\newcommand{\ra}{\rightarrow}

\newcommand{\Ric}{\operatorname{Ric}}

\newcommand{\tr}{\operatorname{tr}}

\newcommand{\divv}{\operatorname{div}}

\begin{document}
\title[Geometric inequalities for free boundary hypersurfaces]{Geometric inequalities for free boundary hypersurfaces in a ball}

\author{Yimin Chen}
\address{Department of Mathematical Sciences, Tsinghua University, Beijing 100084, P.R. China}
\email{\href{mailto:cym16@mails.tsinghua.edu.cn}{cym16@mails.tsinghua.edu.cn}}

\author{Yingxiang Hu}
\address{School of Mathematics, Beihang University, Beijing 100191, P.R. China}
\email{\href{mailto:huyingxiang@buaa.edu.cn}{huyingxiang@buaa.edu.cn}}

\author{Haizhong Li}
\address{Department of Mathematical Sciences, Tsinghua University, Beijing 100084, P.R. China}
\email{\href{mailto:lihz@tsinghua.edu.cn}{lihz@tsinghua.edu.cn}}



\begin{abstract}
	In this paper, we prove a family of sharp geometric inequalities for free boundary hypersurfaces in a ball in space forms.
\end{abstract}

\maketitle

\section{Introduction}
The sharp geometric inequalities for closed hypersurfaces in Riemannian manifolds have attracted a lot of attention during the past few decades. By means of curvature flows, the sharp geometric inequalities such as quermassintegral inequalities, Minkowski type inequalities, Alexandrov-Fenchel type inequalities and more general weighted inequalities have been established for closed hypersurfaces in space forms or warped product spaces, see  \cite{Guan-Li2009,Ge-Wang-Wu2014,Li-Wei-Xiong2014,Wang-Xia2014,BHW2016,Hu-Li2019,Hu-Li2021,Wei-Xiong2015,ACW2021,AHL2020,Hu-Li-Wei2020,Scheuer-Xia2019}, etc. Another powerful approach to prove the sharp geometric inequalities is the (weighted) Reilly formula, which can be used to prove Heintze-Karcher type inequality and Minkowski type inequalities for closed hypersurfaces in Riemannian manifolds, see e.g. \cite{Reilly1977,Ros1988,Qiu-Xia2014,Xia2016,LX19}. 

Influenced by the pioneering works due to De Lellis-M\"uller \cite{DeLellis-Muller2005} and De Lellis-Topping \cite{DeLellis-Topping2010}, Perez \cite[Thm. 3.1]{Perez2011} proved the following geometric inequality for closed hypersurfaces in $\mathbb R^{n+1}$. The equality case was characterized by Cheng and Zhou \cite[Thm. 1.3]{Cheng-Zhou2014}.
\begin{thmA}[\cite{Perez2011,Cheng-Zhou2014}] Let $n\geq 2$ and $\Sigma$ be a closed immersed oriented hypersurface in $\mathbb R^{n+1}$ with nonnegative Ricci curvature. Then
\begin{align}\label{ineq-Perez-1}
\int_{\Sigma} |H-\-H|^2   \leq \frac{n}{n-1} \int_{\Sigma} |\mathring{h}|^2 , 
\end{align}
where $\mathring{h}=h-\frac{H}{n}g$ and $\-H=\int_{\Sigma}H /|\Sigma|$. Equality holds in \eqref{ineq-Perez-1} if and only if it is a sphere.
\end{thmA}
Inequality \eqref{ineq-Perez-1} is equivalent to
\begin{align}\label{ineq-Perez-2}
\int_{\Sigma} \left|h-\frac{\-H}{n}g \right|^2   \leq \frac{n}{n-1} \int_{\Sigma} |\mathring{h}|^2 .
\end{align}
It was pointed out by De Lellis and Topping \cite{DeLellis-Topping2010} that Perez's inequality \eqref{ineq-Perez-1} holds even for the closed hypersurfaces with nonnegative Ricci curvature in Einstein manifolds, see also \cite[Thm. 1.2]{Cheng-Zhou2014}. Recently, Perez's inequality \eqref{ineq-Perez-2} was generalized by Agostiniani, Fogagnolo and Mazzieri \cite[Thm. 6.1]{AFM2019} to the bounded outward minimising open domain with smooth boundary $\Sigma$ in $\mathbb R^3$ by nonlinear potential theory.  

Kwong \cite[Thms. 2.2 \& 4.1]{Kwong2015} generalized Perez's inequality \eqref{ineq-Perez-1} to the higher order mean curvatures for closed immersed submanifolds in space forms. Let $\mathbb M^{n+1}(K)$ be the simply connected space form with constant sectional curvature $K\in \{-1,0,1\}$. Then $\mathbb M^{n+1}(0)=\mathbb R^{n+1}$, $\mathbb M^{n+1}(1)=\mathbb S^{n+1}$ and $\mathbb M^{n+1}(-1)=\mathbb H^{n+1}$. Let $H_k$ be the $k$th mean curvature and $T_{k}$ the $k$th Newton tensor respectively, see \eqref{dfn-mean-curvature} and \eqref{dfn-Newton-tensor} in \S \ref{sec:2}. We denote by $\mathring{T}_k=T_k-\frac{1}{n}(\tr T_k)g$ the tracefree part of $T_k$, in particular, $\mathring{T}_1=-\mathring{h}$. For closed immersed hypersurfaces with nonnegative Ricci curvature in space forms, Kwong's results can be stated as follows.
\begin{thmB}[\cite{Kwong2015}]
	Let $n\geq 2$ and $\Sigma$ be a closed immersed oriented hypersurface in $\mathbb M^{n+1}(K)$ with nonnegative Ricci curvature. Then for $k\in \{1,\cdots,n-1\}$, we have 
\begin{align}\label{ineq-Kwong}
\int_{\Sigma} (H_k-\-H_k)^2  \leq \frac{n(n-1)}{(n-k)^2} \int_{\Sigma}  |\mathring{T}_k|^2 ,
\end{align} 
where $\-H_k=\int_{\Sigma} H_k /|\Sigma|$. Equality is attained at geodesic spheres. Conversely, if $\Sigma$ has positive Ricci curvature, then equality in \eqref{ineq-Kwong} implies that $\Sigma$ is a geodesic sphere.
\end{thmB}

Due to Fraser-Schoen's important work on the first Steklov eigenvalue and minimal free boundary surfaces \cite{Fraser-Schoen2011,Fraser-Schoen2016}, it is natural to establish the sharp geometric inequalities for free boundary hypersurfaces in a ball. In \cite{WX19}, Wang and Xia utilized the (weighted) Reilly formula to prove the Heintze-Karcher type inequality for such hypersurfaces. More recently, by introducing a specifically designed flow, Scheuer, Wang and Xia \cite{Scheuer-Wang-Xia2018} proved a family of Alexandrov-Fenchel type inequalities for free boundary hypersurfaces in a unit Euclidean ball. See also \cite{Wang-Xia2019,Wang-Weng2020}. 

We use $x:M^n \ra B \subset \mathbb M^{n+1}(K)$ to denote an isometric immersion of an oriented $n$-dimensional compact manifold with boundary $\partial M$ into a (closed) geodesic ball $B$ in $\mathbb M^{n+1}(K)$ such that 
$$
x(\operatorname{int} M)\subset \operatorname{int} B \quad \text{and}\quad  x(\partial M) \subset \partial B.
$$ 
The hypersurface $\Sigma=x(M)$ is called {\em free boundary} if it meets the geodesic sphere $\partial B$ orthogonally. It is convenient to use the following models to represent the space forms $\mathbb M^{n+1}(K)$.
\begin{enumerate}[(i)]
	\item If $K=0$, then $(\mathbb M^{n+1}(0),\-g)=(\mathbb R^{n+1},\d)$ and $B=B_R$ is a Euclidean ball of radius $R$, where $\d$ is the Euclidean metric.
	\item If $K=-1$, then we use the Poincar\'e ball model $(\mathbb B^{n+1},e^{2u}\d)$ to represent the hyperbolic space $\mathbb H^{n+1}$, where $\mathbb B^{n+1}$ is the open unit ball centered at the origin in $\mathbb R^{n+1}$ and $e^{2u}=\frac{4}{(1-|x|^2)^2}$. Abuse of notation, here we also use $x$ to denote the position vector with respect to the center of $\mathbb B^{n+1}$ and $|x|=\d(x,x)^{1/2}$ is its Euclidean norm. Let $B=B_R^{\mathbb H}$ be a geodesic ball in $(\mathbb B^{n+1},e^{2u}\d)$ with radius $R\in(0,\infty)$ centered at the center of $\mathbb B^{n+1}$.
	\item If $K=1$, then we use the model $(\mathbb R^{n+1},e^{2u}\d)$ to represent the unit sphere without the south pole $\mathbb S^{n+1}\backslash \{\mathcal{S}\}$ and $e^{2u}=\frac{4}{(1+|x|^2)^2}$. Let $B=B_R^{\mathbb S}$ be a geodesic ball in $\mathbb S^{n+1}$ with radius $R\in (0,\pi)$ centered at the north pole.
\end{enumerate}
For simplicity, we call a point $p\in \Sigma$ is an {\em elliptic point}, if its second fundamental form satisfies $h_{ij}>cg_{ij}$ at this point, where $c=0$ when $K=0,1$ and $c=1$ when $K=-1$.

Our first result of this paper is the following.
\begin{thm}\label{thm-I}
	Let $n\geq 2$ and $\Sigma$ be a compact immersed oriented hypersurface with free boundary in a ball in space forms. Assume that $\Sigma$ has nonnegative Ricci curvature. Then for $k\in \{1,\cdots,n-1\}$, we have 
	\begin{align}\label{s1:1.2}
	\int_{\Sigma} |H_k-\overline{H}_k |^2 \leq \frac{n(n-1)}{(n-k)^2}\int_{\Sigma} |\mathring{T}_k|^2.
	\end{align}
	Equality is attained at spherical caps or a flat disk. Conversely, if there exists an elliptic point in the interior of $\Sigma$, then equality in \eqref{s1:1.2} implies that $\Sigma$ is a spherical cap.
\end{thm}

If we further assume that $x:(M^n,g)\ra B\subset \mathbb R^{n+1}$ is an isometric embedding, then the free boundary hypersurface $\Sigma$ is called {\em convex} if its second fundamental form with respect to the unit outward normal $\nu$ is positive semi-definite, i.e., $h_{ij}\geq 0$. We obtain the following result for convex free boundary hypersurfaces in a Euclidean ball.
\begin{cor}\label{cor-convex-hypersurface}
	Let $n\geq 2$ and $\Sigma$ be a convex free boundary hypersurface in a ball $B \subset \mathbb R^{n+1}$. Then for  $k\in \{1,\cdots,n-1\}$, we have 
	\begin{align}\label{s1:1.3}
	\int_{\Sigma} |H_k-\overline{H}_k |^2 \leq \frac{n(n-1)}{(n-k)^2}\int_{\Sigma} |\mathring{T}_k|^2.
	\end{align}
	Equality holds in \eqref{s1:1.3} if and only if $\Sigma$ is a spherical cap or a flat disk.	
\end{cor}

In \cite{Scheuer-Wang-Xia2018}, Scheuer, Wang and Xia found a natural counterpart of quermassintegrals $W_k$ for convex free boundary hypersurfaces $\Sigma$ in the unit Euclidean ball. Let $\partial \Sigma$ be the boundary of $\Sigma$ in $\mathbb S^n$. We use $|\partial \Sigma|$ to denote the $(n-1)$-dimensional Hausdorff measure of $\partial \Sigma$, and $|\Sigma|$ to denote the $n$-dimensional Hausdorff measure of $\Sigma$. As an application of Corollary \ref{cor-convex-hypersurface}, we obtain the following geometric inequalities.
\begin{cor}\label{cor-low-dim}
	\begin{enumerate}[(i)]
		\item Let $\Sigma^2$ be a convex free boundary surface in $\-{\mathbb B}^{3}$. Then 
	\begin{align}\label{1.5}
	 \frac{1}{4}\frac{(\int_{\Sigma}H)^2}{|\Sigma|}+|\partial \Sigma| \geq 2\pi.
	\end{align} 
Equality holds in \eqref{1.5} if and only if it is a flat disk or a spherical cap.
	    \item Let $\Sigma^3$ be a convex free boundary hypersurface in $\-{\mathbb B}^{4}$. Then 
	    \begin{align}\label{1.7}
	    \frac{1}{12}\(\frac{1}{3}\frac{(\int_{\Sigma}H)^2}{|\Sigma|}+|\partial \Sigma|\) \geq f_3\circ f_1^{-1}\(\frac{1}{4}|\Sigma|\),
	    \end{align} 
	    where $f_k(r)=W_k(\hat{C}_r)$ is a strictly increasing function of $r$ (see \S \ref{sec:2} for details), and $f_l^{-1}$ is the inverse of $f_l$. Equality holds in \eqref{1.7} if and only if it is a flat disk or a spherical cap. 	
\end{enumerate}   
\end{cor}

\begin{rem}
	\begin{enumerate}
    \item By H\"older inequality, inequality \eqref{1.5} implies the Willmore type inequality
	\begin{align*}
	\frac{1}{4}\int_{\Sigma}H^2+|\partial \Sigma| \geq 2\pi,
	\end{align*} 
    which was proved by Volkmann \cite[Cor. 5.8]{Volkmann2016}.
    \item By H\"older inequality, inequality \eqref{1.7} implies the Willmore type inequality 
    \begin{align*}
    \frac{1}{12}\(\frac{1}{3}\int_{\Sigma}H^2+|\partial \Sigma|\) \geq f_3\circ f_1^{-1}\(\frac{1}{4}|\Sigma|\),
    \end{align*}
    which was proved by Scheuer, Wang and Xia \cite[Cor. 1.4]{Scheuer-Wang-Xia2018}. 
\end{enumerate}
\end{rem} 

To further generalize the inequality \eqref{s1:1.2}, we consider a Riemannian triple $(\Sigma,g,V)$ which constitutes an $n$-dimensional smooth connected manifold $\Sigma$, a Riemannian metric $g$ and a smooth nonzero function $V$ on $\Sigma$. $V$ is called a {\em potential function}. A Riemannian triple $(\Sigma,g,V)$ is called {\em sub-static} if 
	\begin{align}\label{s1:dfn-sub-static}
	\D V g -\nabla^2 V+V \operatorname{Ric}^{\Sigma} \geq 0, \quad \text{on $\Sigma$},
	\end{align}
	where $\D$, $\nabla^2$ and $\operatorname{Ric}^{\Sigma}$ are the Laplacian, Hessian and Ricci curvature of $(\Sigma,g)$, respectively. We also call it {\em strictly sub-static} if the strict inequality holds in \eqref{s1:dfn-sub-static} on $\Sigma$. The potential functions associated to free boundary hypersurfaces in a ball in space forms are the functions $V_a$ which are defined as follows. 
\begin{equation*}
V_a=\left\{\begin{array}{ll}
\langle x,a\rangle,&\qquad \text{if $K=0$};\\
\frac{2\langle x,a\rangle}{1-|x|^2},&\qquad \text{if $K=-1$};\\
\frac{2\langle x,a\rangle}{1+|x|^2},&\qquad \text{if $K=1$},
\end{array}
\right.
\end{equation*}
where $a \in \mathbb R^{n+1}$ is a constant vector and $\langle x,a \rangle=\d(x,a)$. For simplicity, we use $B_{a+}=\{V_a>0\}\cap B$ to denote the half ball of $B$ in the direction of $a$. We obtain the following weighted geometric inequalities.
\begin{thm}\label{thm-II}
	Let $n\geq 2$ and $x:M \ra \mathbb M^{n+1}(K)$ be a compact immersed oriented hypersurface $\Sigma$ with free boundary in a ball $B$. Assume $\Sigma$ lies in $B_{a+}$ and $(\Sigma,g,V_a)$ is sub-static. Then for $k\in \{1,\cdots,n-1\}$, we have 
	\begin{align}\label{s1:1.1}
	\int_{\Sigma}  V_a|H_k-\overline{H}_k^{V_a}|^2 \leq \frac{n(n-1)}{(n-k)^2}\int_{\Sigma} V_a |\mathring{T}_k|^2,
	\end{align}
	where $\overline{H}_k^{V_a}=\int_{\Sigma}V_a H_k/\int_{\Sigma}V_a$. Equality is attained at spherical caps. Conversely, if $(\Sigma,g,V_a)$ is strictly sub-static and there exists an elliptic point in the interior of $\Sigma$, then equality in \eqref{s1:1.1} implies that $\Sigma$ is a spherical cap. 
\end{thm}

As a direct application, we obtain the following corollary.
\begin{cor}\label{cor-weighted-free-boundary-Euclidean}
	Let $n\geq 2$ and $x:M \ra \mathbb R^{n+1}$ be a smooth embedded oriented hypersurface $\Sigma$ with free boundary in a ball $B$. If $\Sigma$ is strictly convex, then there exists a unit vector $a\in\mathbb R^{n+1}$ such that $V_a>0$ and for $k\in \{1,\cdots,n-1\}$, we have
	\begin{align*}
	\int_{\Sigma}  V_a|H_k-\overline{H}_k^{V_a}|^2 \leq \frac{n(n-1)}{(n-k)^2}\int_{\Sigma} V_a |\mathring{T}_k|^2.
	\end{align*}
	Equality holds if and only if $\Sigma$ is a spherical cap. 
\end{cor}

The paper is organized as follows. In \S \ref{sec:2}, we collect basic facts about free boundary hypersurfaces and weighted Reilly formula. In \S \ref{sec:3}, we give the proof of Theorems \ref{thm-I}, \ref{thm-II} and Corollaries \ref{cor-convex-hypersurface}, \ref{cor-low-dim}, \ref{cor-weighted-free-boundary-Euclidean}. 
\begin{ack}
	Y. Chen and H. Li was partially supported by NSFC Grant No.11831005 and NSFC-FWO grant No.1196131001. Y. Hu was partially supported by NSFC Grant No.12101027. 
\end{ack}

\section{Preliminaries}\label{sec:2}
Let $n\ge 2$. Let $(\Sigma,g)$ be a free boundary hypersurface given by an immersion $x:M^n \ra B \subset \mathbb M^{n+1}(K)$. We denote by $\-\nabla$, $\-\D$ and $\-\nabla^2$ the gradient, the Laplacian and the Hessian on $(\mathbb M^{n+1}(K),\-g)$ respectively, while by $\nabla$, $\D$ and $\nabla^2$ the gradient, the Laplacian and the Hessian on $(\Sigma,g)$ respectively. We choose one of the unit normal vector field along the immersion $x$ and denote it by $\nu$. Denote by $h$ and $H$ the second fundamental form and the mean curvature of the hypersurface $\Sigma$ respectively. Precisely, $h(X,Y)=\-g(\-\nabla_X \nu,Y)$ and $H=\tr h$. Since the mean curvature vector is independent of the choice of $\nu$, we make a convention on $\nu$ to be the opposite direction of mean curvature vector. We denote $\mu$ the outward unit normal of $\partial \Sigma$ in $\Sigma$, $\-N$ the outward normal of $\partial B$ in $B$ and $\-\nu$ the outward normal of $\partial \Sigma$ in $\partial B$, respectively. If $\Sigma$ is a free boundary hypersurface in $B$, then $\mu=\-N$ and $\-\nu=\nu$ along the boundary $\partial \Sigma$. See Figure \ref{fig} below.
\begin{figure}[htbp]
	\begin{center}
		\begin{tikzpicture}[scale=3]
		\draw (0,0) circle (1);
		\draw (-0.866,0.5) arc[start angle=-120, end angle=-60, radius=1.732]; 
		\draw[-stealth] (0,0.268) --(0,0) node[right] {$\nu$};
		\draw[-stealth] (1,0.65) -- (0.866,0.5);
		\node at (1,0.72) {$\partial\Sigma$};
		\node at (0.5,0.25) {$\Sigma$};
		\node at (0.5,1) {$\partial B$};
		\draw[-stealth] (-0.866,0.5)--(-1.083,0.625) node[above] {$\mu=\overline{N}$};
		\draw[-stealth] (-0.866,0.5)--(-0.991,0.278) node[left] {$\overline{\nu}=\nu$};
		\end{tikzpicture}
	\end{center}
	\caption{}
	\label{fig}
\end{figure}

The principal curvatures $\k=(\k_1,\cdots,\k_n)$ are the eigenvalues of the Weingarten matrix $\mathcal{W}=(h_i^j)=(g^{jk}h_{ki})$ on $\Sigma$, i.e., the eigenvalues of the second fundamental form $h$ with respect to the induced metric $g$. For $k\in \{1,2,\cdots,n\}$, the $k$-th mean curvature $H_k$ is defined by 
\begin{align}\label{dfn-mean-curvature}
H_k=\frac{1}{k!}\sum_{\substack{1\leq i_1,\cdots,i_k\leq n\\1\leq j_1,\cdots,j_k\leq n}} \d^{j_1j_2\cdots j_k}_{i_1i_2\cdots i_k} h^{i_1}_{j_1} h^{i_2}_{j_2}\cdots h^{i_k}_{j_k}=\sum_{i_1<\cdots<i_k}\k_{i_1}\cdots \k_{i_k},
\end{align} 
where $\d_{j_1\cdots j_k}^{i_1 \cdots i_k}$ is the generalized Kronecker delta function. For $m\in \{0,1,\cdots,n-1\}$, the $m$-th Newton tensor of $\mathcal{W}=(h_i^j)$ is defined by
\begin{align}\label{dfn-Newton-tensor}
(T_{m})^j_i=\frac{1}{m!}\sum_{\substack{1\leq i_1,\cdots,i_{m}\leq n\\1\leq j_1,\cdots,j_{m}\leq n}} \d^{j j_1j_2\cdots j_m}_{i i_1i_2\cdots i_m} h^{i_1}_{j_1} h^{i_2}_{j_2}\cdots h^{i_{m}}_{j_{m}}.
\end{align}
By Codazzi equations for hypersurfaces in space forms, we have 
\begin{lem}[\cite{Reilly1973}]\label{s2:newton-tensor-property} Let $\Sigma$ be an immersed hypersurface in space forms. Let $\divv$ be the divergence on $\Sigma$. Then for $m\in \{0,\cdots,n-1\}$, we have
	\begin{align}\label{divengence-Netwton-tensor}
	\divv(T_{m})=0, \quad \operatorname{tr}(T_{m})=(n-m)H_{m}.
	\end{align} 
\end{lem}
Let $\mathring{T}_m=T_m-\frac{1}{n}\operatorname{tr}(T_{m}) g$. It follows from \eqref{divengence-Netwton-tensor} that 
\begin{align}\label{divengence-tracefree-Newton-tensor}
\divv(\mathring{T}_m)=-\frac{n-m}{n}\nabla H_m.
\end{align}

We collect the basic facts on $H_k$ (see e.g. \cite{Hardy-1934}). The $k$-th Garding cone is defined by $\G_k^{+}=\{\k \in \mathbb R^{n} ~|~ H_i(\k)>0, i=1,\cdots,k\}$.
\begin{lem} Let $k\in \{1,2,\cdots,n-1\}$. If $\k\in \G_k^{+}$, the following Newton-MacLaurin inequality holds
	\begin{align}\label{s2:Newton-MacLaurin-ineq}
	\frac{n-k}{n}H_{1}H_{k}\geq (k+1) H_{k+1}.
	\end{align}
	Equality holds in \eqref{s2:Newton-MacLaurin-ineq} if and only if $\k=cI$ for some constant $c>0$, where $I=(1,\cdots,1)$.
\end{lem}

It is clear that if $\mathring{T}_1=0$, then $\mathring{T}_1=-\mathring{h}$ which implies $\k=cI$. For $k\in \{2,\cdots,n\}$, the following lemma characterizes the traceless property of $T_k$.
\begin{lem}\label{s2:lem-Tk-umbilicity}
	Let $k\in \{2,\cdots,n-1\}$. If $\k\in \G_{k}^{+}$, then $\mathring{T}_k=0$ if and only if $\k=cI$ for some constant $c>0$.
\end{lem}
\begin{proof}
	If $\mathring{T}_k=0$, we have $(T_k)_j^i=\frac{n-k}{n}H_k\d_j^i$. Multiplying by $h_i^j$ on the both sides, we get
	\begin{align*}
	(k+1)H_{k+1}=\frac{n-k}{n}H_k H_1,
	\end{align*}
	which is an equality in \eqref{s2:Newton-MacLaurin-ineq}. If $\k\in \G_k^{+}$, then $\k=cI$ for some constant $c>0$. The converse is obvious.
\end{proof}

The following proposition is a well-known fact for free boundary hypersurfaces in a geodesic ball in space forms.
\begin{lem}[\cite{Ros-Souam1997,Li-Xiong2018,WX19}]
	Let $x:M\ra \mathbb R^{n+1}$ be an immersed oriented hypersurface $\Sigma$ with free boundary in a ball $B$. Then $\mu$ is a principal direction of $\partial \Sigma$ in $\Sigma$. Namely, 
	\begin{align*}
	h(e,\mu)=0, \quad  \text{for any $e\in T(\partial \Sigma)$.}
	\end{align*}
	In turn, $\-\nabla_\mu \nu=h(\mu,\mu)\mu$. Moreover, for any tangential vector field $Z\in T(\partial \Sigma)$, we have
	\begin{align}\label{s2:T_k-principal-direction}
	T_k(\mu,Z)=0, \quad k\in\{0,1,\cdots,n-1\}.
	\end{align}
\end{lem}

In \cite{Scheuer-Wang-Xia2018}, Scheuer, Wang and Xia defined the following geometric quantities for free boundary hypersurfaces in a unit Euclidean ball, which are expected to be the counterparts to the quermassintegrals for closed hypersurfaces in space forms: 
\begin{align*}
W_0(\hat{\Sigma})=&|\hat{\Sigma}|, \qquad W_1(\hat{\Sigma})=\frac{1}{n+1}|\Sigma|,\\
W_k(\hat{\Sigma})=&\frac{1}{n+1}\int_{\Sigma}\frac{H_{k-1}}{\binom{n}{k-1}}+\frac{k-1}{(n+1)(n-k+2)}W_{k-2}^{\mathbb S}(\widehat{\partial \Sigma}),\quad 2\leq k \leq n+1,
\end{align*}
where $|\hat{\Sigma}|$ denotes the $(n+1)$-dimensional Hausdorff measure of $\hat{\Sigma}$, and $W_{k-2}^{\mathbb S}(\widehat{\partial \Sigma})$ denotes the $(k-2)$-th quermassintegral of the closed hypersurface $\partial \Sigma\subset \mathbb S^n$, see \cite[(1.1)]{Scheuer-Wang-Xia2018}. Furthermore, the spherical cap of radius $R$ around $a\in \mathbb S^n$ is defined by
$$
C_R(a)=\{y \in \-{\mathbb B}^{n+1} :~|y-\sqrt{R^2+1}~a|=R\}, \quad R<\infty,
$$
and $\hat{C}_R(a)$ is the convex domain enclosed by the spherical cap $C_R(a)$ in $\-{\mathbb B}^{n+1}$, where the argument $a$ will be dropped in cases where it is not relevant. They proved that if $\Sigma$ is a convex hypersurface with free boundary in the unit ball $\-{\mathbb B}^{n+1}$, then 
\begin{align}\label{quermassintegral-equality}
W_{n+1}(\hat{\Sigma})=\frac{\omega_n}{2(n+1)},
\end{align}
and for $k\in \{0,1,\cdots,n-1\}$, there holds
\begin{align}\label{quermassintegral-inequality}
W_{n}(\hat{\Sigma}) \geq (f_n \circ f_k^{-1}) (W_k(\hat{\Sigma})),
\end{align}
where $f_k=f_k(r)$ is the strictly increasing function $f_k(r)=W_k(\hat{C}_r)$, where $\hat{C}_r$ is the convex domain enclosed by the spherical cap $C_r$ in $\-{\mathbb B}^{n+1}$. Equality holds if and only if $\Sigma$ is a spherical cap or a flat disk. 

In particular, if $n=2$,
\begin{align*}
W_3(\hat{\Sigma})=\frac{1}{3}\int_{\Sigma}H_{2}+\frac{2}{3}W_{1}^{\mathbb S}(\widehat{\partial \Sigma})=\frac{1}{3}\int_{\Sigma}H_{2}+\frac{1}{3}|\partial \Sigma|,
\end{align*}
then the equality \eqref{quermassintegral-equality} becomes
\begin{align}\label{s2:n=2-case}
\frac{1}{3}\int_{\Sigma}H_{2}+\frac{1}{3}|\partial \Sigma|=\frac{\omega_2}{6}=\frac{2\pi}{3}.
\end{align}
If $n=3$, 
\begin{align*}
W_3(\hat{\Sigma})=\frac{1}{12}\int_{\Sigma}H_{2}+\frac{1}{4}W_{1}^{\mathbb S}(\widehat{\partial \Sigma})=\frac{1}{12}\int_{\Sigma}H_{2}+\frac{1}{12}|\partial \Sigma|,\quad W_1(\hat{\Sigma})=\frac{1}{4}|\Sigma|,
\end{align*}
then the inequality \eqref{quermassintegral-inequality} with $k=1$ becomes
\begin{align}\label{s2:n=3-case}
\frac{1}{12}\(\int_{\Sigma}H_{2}+|\partial \Sigma|\)\geq (f_3 \circ f_1^{-1}) \(\frac{1}{4}|\Sigma|\).
\end{align}

The main tool in the proof is the weighted Reilly formula. Reilly \cite{Reilly1977} first established an important integral formula on hypersurfaces in Riemannian manifolds, which is very useful for manifolds with nonnegative Ricci curvature. Qiu and Xia \cite{Qiu-Xia2014} established a weighted Reilly formula which can be used for manifolds with sectional curvature bounded from below. Recently, Li and Xia \cite{LX19} further generalized the Reilly formula and using this new integral formula to prove geometric inequalities in warped product spaces. Abuse of notations, we use $\-\nabla$, $\-\D$ and $\-\nabla^2$ to denote the gradient, the Laplacian and the Hessian on $(\Omega,\-g)$ respectively, while we use $\nabla$, $\D$ and $\nabla^2$ to denote the gradient, the Laplacian and the Hessian on $(\partial \Omega,g)$ respectively. In this paper, we will use the following special form of the weighted Reilly formula.
\begin{prop}\cite[Cor. 5.1]{LX19}
	Let $(\Omega,\-g)$ be a compact Riemannian manifold with smooth connected boundary $\partial\Omega$. Let $V\in C^{\infty}(\overline{\Omega})$ be a smooth positive function on $\-\Omega$ such that $\frac{\-\nabla^2 V}{V}$ is continuous up to $\partial\Omega$. Then for any $f\in C^{\infty}(\overline{\Omega})$, the following integral identity holds:
	\begin{align}\label{s4:Li-Xia-formula}
	&\int_{\Omega} V \(\-\D f-\frac{\-\D V}{V}f\)^2-V\left|\-\nabla^2 f-\frac{\-\nabla^2 V}{V}f \right|^2 d\Omega \nonumber \\
	=&\int_{\Omega} \(\-\D V\-g-\-\nabla^2 V+V\operatorname{Ric}\)\(\-\nabla f-\frac{\-\nabla V}{V}f,\-\nabla f-\frac{\-\nabla V}{V}f\) d\Omega\nonumber \\
	&+\int_{\partial\Omega} V\(h-\frac{V_{\nu}}{V}g\)\( \nabla f -\frac{\nabla V}{V}f,\nabla f-\frac{\nabla V}{V}f \) dA \nonumber\\
	&+\int_{\partial\Omega} VH\( f_{\nu}-\frac{V_{\nu}}{V}f\)^2+2V \( f_{\nu}-\frac{V_{\nu}}{V}f\)\(\D f-\frac{\D V}{V}f \) dA.
	\end{align}
	Here $\nu$ is the unit outward normal of $\partial\Omega$ and $V_{\nu}=\-\nabla_{\nu}V$, $f_{\nu}=\-\nabla_{\nu}f$, $h(\cdot,\cdot)$ and $H$ are the second fundamental form and the mean curvature of $\partial\Omega$, and $\Ric$ is Ricci curvature of $\Omega$ respectively. 
\end{prop}

\section{Proof of Main Results}\label{sec:3}
\begin{proof}[Proof of Theorem \ref{thm-I}]
	We first give the proof of the inequality \eqref{s1:1.2} in Theorem \ref{thm-I} and the inequality \eqref{s1:1.1} in Theorem \ref{thm-II} in a unified way.
	We choose $V=1$ in the proof of Theorem \ref{thm-I}, $V=V_a$ in the proof of Theorem \ref{thm-II} respectively. In Theorem \ref{thm-II}, the hypersurface $\Sigma$ is assumed to be contained in an open half-ball $B_{a+}$, which is equivalent to $V_a>0$.
	
	Consider the Neumann boundary value problem 
	\begin{equation}\label{Neumann-eq}
	\left\{\begin{aligned}
	\D f-\frac{\D V}{V}f=&H_k-\-H_k^{V}, \quad &\text{in $\Sigma$},\\
	f_{\mu}-\frac{V_{\mu}}{V}f=&0, \quad &\text{on $\partial \Sigma$},
	\end{aligned}\right.
	\end{equation}
	where $\mu$ is the unit outward normal of $\partial \Sigma$ in $\Sigma$. The existence and uniqueness (up to an additive $\a V$) of the solution to \eqref{Neumann-eq} follows from the Fredholm alternative. For simplicity, we denote by
	$$
	A_{ij}=f_{ij}-\frac{V_{ij}}{V}f, \quad A=\D f-\frac{\D V}{V}f.
	$$
	Let $\mathring{A}_{ij}=A_{ij}-\frac{A}{n}g_{ij}$ be the tracefree part of $A_{ij}$.
	
	Using Eq. \eqref{Neumann-eq}, we have
	\begin{align}\label{s3:key-step}
	\int_{\Sigma} V (H_k-\-H_k^{V})^2 =&\int_{\Sigma} (H_k-\-H_k^{V}) (V \D f-f \D V) \nonumber\\
	=&\int_{\partial \Sigma} (H_k-\-H_k^{V})(V f_{\mu}-f V_{\mu})-\int_{\Sigma} g(\nabla H_k, V\nabla f-f\nabla V) \nonumber\\
	=&\frac{n}{n-k}\int_{\Sigma} g(\divv(\mathring{T}_k), V\nabla f-f\nabla V) \nonumber\\
	=&\frac{n}{n-k}\(\int_{\partial\Sigma} \mathring{T}_k(V \nabla f-f\nabla V,\mu)-\int_{\Sigma} V\sum_{i,j} (\mathring{T}_k)_{ij}\mathring{A}_{ij}\)\nonumber\\
	\leq &\frac{n}{n-k} \(\int_{\Sigma} V |\mathring{T}_k|^2\)^{1/2} \(\int_{\Sigma}V \sum_{i,j}|\mathring{A}_{ij}|^2\)^{1/2}.
	\end{align}
	In the second equality we used integration by parts, and in the third equality we used $f_{\mu}-\frac{V_{\mu}}{V}f=0$ to eliminate the boundary integral on $\partial \Sigma$ and \eqref{divengence-tracefree-Newton-tensor}. In the fourth equality we used integration by parts again. In the last line, the boundary integral on $\partial \Sigma$ vanishes due to \eqref{s2:T_k-principal-direction} and the inequality follows from the H\"older inequality.
	
	We use $\~\nabla$ to denote the gradient on $\partial\Sigma$. We take $\Omega=\Sigma$ and $V=1$ or $V=V_a$ in the weighted Reilly formula \eqref{s4:Li-Xia-formula}, in view of the boundary condition $f_{\mu}-\frac{V_{\mu}}{V}f=0$, we obtain
	\begin{align}\label{s3:Reilly-formula}
	\int_{\Sigma} V \( A^2-\sum_{i,j}|A_{ij}|^2\) 	=&\int_{\partial\Sigma} V \(h^{\partial\Sigma}-\frac{V_{\mu}}{V}g^{\partial\Sigma}\)\( \~\nabla f -\frac{\~\nabla V}{V}f,\~\nabla f-\frac{\~\nabla V}{V}f \) \nonumber\\
	&+\int_{\Sigma} \(\D V g- \nabla^2 V+V\operatorname{Ric}^{\Sigma}\)\(\nabla f-\frac{\nabla V}{V}f,\nabla f-\frac{\nabla V}{V}f\).
	\end{align}
	
	By the free boundary condition, we have $\mu=\-N$. Then the induced metric $g^{\partial \Sigma}$ from $(\Sigma,g)$ coincides with the metric $g^{\partial B}|_{T(\partial \Sigma)\otimes T(\partial \Sigma)}$, where $g^{\partial B}$ is the metric of $\partial B$. Furthermore, the second fundamental form $h^{\partial \Sigma}$ coincides with $h^{\partial B}|_{T(\partial \Sigma)\otimes T(\partial \Sigma)}$, where $h^{\partial B}$ is the second fundamental form of $ \partial B$ given by
	\begin{align}\label{s3:key-geodesic-sphere}
	h^{\partial B}=\left\{\begin{aligned}&(1/R)g^{\partial B}, \quad &\text{if $K=0$},\\
	&(\coth R) g^{\partial B}, \quad &\text{if $K=-1$},\\
	&(\cot R)  g^{\partial B}, \quad &\text{if $K=1$}.
	\end{aligned} 
	\right.
	\end{align} 
    We have the following identity
	\begin{align}\label{s3:key-id}
	h^{\partial B}-\frac{(V_a)_{\-N}}{V_a}g^{\partial B}=0, \quad \text{on $\partial B$},
	\end{align}
	see \cite[(5.12)]{WX19}. For convenience of readers, we give a proof of \eqref{s3:key-id} in Remark \ref{remark-identity}.
	
	In case of Theorem \ref{thm-I}, $\D V g- \nabla^2 V+V\operatorname{Ric}^{\Sigma}=\operatorname{Ric}^{\Sigma}\geq 0$ and $h^{\partial\Sigma}-\frac{V_{\mu}}{V}g^{\partial\Sigma}=h^{\partial\Sigma}\geq 0$.
	In case of Theorem \ref{thm-II}, $\D V_a g-\nabla^2 V_a+V_a \operatorname{Ric}^{\Sigma}\geq 0$ and $h^{\partial \Sigma}-\frac{(V_a)_{\mu}}{V_a}g^{\partial \Sigma}=0$ due to \eqref{s3:key-id}. It follows from \eqref{s3:Reilly-formula} that $\int_{\Sigma} V A^2 - V \sum_{i,j}|A_{ij}|^2 \geq 0$, which yields
	\begin{align}\label{s3:key-inequality}
	\int_{\Sigma} V \sum_{i,j}|\mathring{A}_{ij}|^2 \leq \frac{n-1}{n}\int_{\Sigma} V A^2.
	\end{align}
	Substituting \eqref{s3:key-inequality} into \eqref{s3:key-step}, we obtain
	\begin{align}\label{s3:key-ineq-last-step}
	\int_{\Sigma} V (H_k-\-H_k^{V})^2  \leq &\frac{n}{n-k} \(\int_{\Sigma} V |\mathring{T}_k|^2\)^{\frac{1}{2}} \(\frac{n-1}{n}\int_{\Sigma}V A^2\)^\frac{1}{2} \nonumber\\
	\leq &\frac{n}{n-k} \sqrt{\frac{n-1}{n}} \(\int_{\Sigma} V |\mathring{T}_k|^2\)^{\frac{1}{2}}\(\int_{\Sigma}V |H_k-\-H_k^{V}|^2 \)^\frac{1}{2}.
	\end{align}
	Finally, if $\int_{\Sigma}V |H_k-\-H_k^{V}|^2=0$, then the inequality \eqref{s1:1.2} or \eqref{s1:1.1} holds trivially. Otherwise, by eliminating $\int_{\Sigma}V |H_k-\-H_k^{V}|^2$ on both sides of \eqref{s3:key-ineq-last-step}, we obtain 
	\begin{align*}
	\int_{\Sigma} V (H_k-\-H_k^{V})^2 \leq \frac{n(n-1)}{(n-k)^2}\int_{\Sigma} V|\mathring{T}_k|^2.
	\end{align*}
	
	To complete the proof of Theorem \ref{thm-I}, we define the subset $\Sigma_{+}$ of $\Sigma$ which consists of elliptic points in $\Sigma$, i.e.,
	\begin{align*}
	\Sigma_{+}=\{x\in \operatorname{int} \Sigma ~:~h_{ij}>cg_{ij}\},
	\end{align*} 
	where $c=0$ when $K=0,1$ and $c=1$ when $K=-1$. By assumption, $\Sigma_{+}$ is nonempty and it is obviously open. Now we show that $\Sigma_{+}$ is also closed, by showing that
	\begin{align*}
	h_{ij}|_{\Sigma_{+}} \geq (c+\e) g_{ij}|_{\Sigma_{+}}
	\end{align*} 
	for some constant $\e>0$. 
	
	In view of \eqref{s3:Reilly-formula}, the equality in \eqref{s3:key-inequality} implies that
	\begin{align}\label{s3:Equality-I}
	\operatorname{Ric}^{\Sigma}\(\nabla f,\nabla f\)=0.
	\end{align}
	On the other hand, the equality in \eqref{s3:key-step} implies that there exists a constant $\b\in\mathbb R$ such that 
	\begin{align}\label{s3:Equality-II}
	(\mathring{T}_k)_{ij}=\b \mathring{A}_{ij}=\b (f_{ij}-\frac{\D f}{n}g_{ij}).
	\end{align}
	Let $p \in \Sigma_{+}$, then by Gauss equation
	$$
	\operatorname{Ric}^{\Sigma}_{ij}=(H h_{ij}-h_{ik}h^k_j)+(n-1)Kg_{ij}>0,
	$$
	we have $\operatorname{Ric}^{\Sigma}>0$ in an open neighborhood $U\subset \Sigma_{+}$ of $p$. Then by \eqref{s3:Equality-I}, we have $f\equiv \a$ for some $\a\in \mathbb R$ and hence $H_k\equiv \-H_k>H_k(cI)$ in $U$ by Eq. \eqref{Neumann-eq}. So $\mathring{T}_k\equiv 0$ in $U$ by \eqref{s3:Equality-II}, which implies
	$$
	h_{ij}=\frac{1}{\binom{n}{k}}H_k g_{ij}=\frac{1}{\binom{n}{k}}\-H_k g_{ij}, \quad \text{in $U$}.
	$$ 
	due to Lemma \ref{s2:lem-Tk-umbilicity}. Thus, we have
	$$
	h_{ij}\geq (c+\e)g_{ij}, \quad \text{in $\-U$}.
	$$
	Therefore, $\Sigma_{+}$ is closed and $\Sigma_{+}=\operatorname{int}\Sigma$ by connectedness. It follows that $\Sigma$ is umbilical and hence it is a spherical cap. The proof of Theorem \ref{thm-I} is completed. 	
\end{proof}

\begin{rem}\label{remark-identity}
	Since $\partial B$ is a geodesic sphere of radius $R$ in $\mathbb M^{n+1}(K)$, we have
	\begin{align}\label{s3:normal-vector}
	\-N=\left\{\begin{aligned}&x/R, \quad &\text{if $K=0$},\\
	&x/\sinh R, \quad &\text{if $K=-1$},\\
	&x/\sin R, \quad &\text{if $K=1$}.
	\end{aligned} 
	\right.
	\end{align}
	On the other hand, $\partial B$ can be viewed as the Euclidean sphere $|x|=R_{\mathbb R}$, where $R_{\mathbb R}$ is given by 
	\begin{align}\label{defn-R_R}
	R_{\mathbb R}=\left\{\begin{aligned}
	&R,\quad   &K=0; \\
	&\sqrt{\frac{\cosh R-1}{\cosh R+1}}, \quad  &K=-1;\\
	&\sqrt{\frac{1-\cos R}{1+\cos R}}, \quad &K=1.
	\end{aligned} \right.
	\end{align}
	Let $\{E_i\}_{i=1}^{n+1}$ be an orthonormal basis in $(\mathbb R^{n+1},\d)$ and $\-E_i=e^{-u}E_i$. Then $\{\-E_i\}_{i=1}^{n+1}$ is an orthonormal basis in $(\mathbb M^{n+1}(K),\-g=e^{2u}\d)$. In view of the models of the space forms, any function $f$ defined on $(\mathbb M^{n+1}(K),\-g)$ can be also considered as a function defined on $\mathbb B^{n+1}$ or $\mathbb R^{n+1}$ respectively. Then there holds 
	$$
	\nabla^{\mathbb R}f=\sum_{i=1}^{n+1}E_i(f)E_i=e^{2u}\-E_i(f)\-E_i=e^{2u}\-\nabla f.
	$$
	Then it follows from \eqref{s3:normal-vector} and \eqref{defn-R_R} that
	\begin{align}\label{s3:key-identity}
	\frac{(V_a)_{\-N}}{V_a}=\-g(\-\nabla\log V_a, \-N)=\langle \nabla^{\mathbb R}\log V_a, \-N\rangle=\left\{\begin{aligned}&\frac{\langle N,a\rangle}{\langle x,a\rangle}=\frac{1}{R}, \quad &\text{if $K=0$},\\
	&\frac{\langle N,a\rangle}{\langle x,a\rangle}+\frac{2\langle x,\-N\rangle}{1-|x|^2}=\coth R, \quad &\text{if $K=-1$},\\
	&\frac{\langle N,a\rangle}{\langle x,a\rangle}-\frac{2\langle x,\-N\rangle}{1+|x|^2}=\cot R, \quad &\text{if $K=1$}.
	\end{aligned} \right.
	\end{align}
	Then the identity \eqref{s3:key-id} follows from \eqref{s3:key-geodesic-sphere} and \eqref{s3:key-identity}.
\end{rem}

As a direct application, we give the proof of Corollaries \ref{cor-convex-hypersurface}, \ref{cor-low-dim}.
\begin{proof}[Proof of Corollary \ref{cor-convex-hypersurface}]
	For convex hypersurfaces with free boundary in a ball of $\mathbb R^{n+1}$, the nonnegativity of Ricci curvature follows directly from Gauss equation. Then the inequality follows immediately. To prove the rigidity part, we may assume that $\Sigma$ is not the flat disk, otherwise we are done. Then by \cite[Lem. 3.1]{Lambert-Scheuer2017}, there exists a strictly convex point in the interior of $\Sigma$. Then it follows from Theorem \ref{thm-I} that $\Sigma$ is a spherical cap. This completes the proof of Corollary \ref{cor-convex-hypersurface}.
\end{proof}

\begin{proof}[Proof of Corollary \ref{cor-low-dim}]
	Inequality \eqref{s1:1.3} with $k=1$ is equivalent to
	\begin{align}\label{s3:key-estimate}
	\int_{\Sigma}H_2\leq \frac{n-1}{2n}\frac{(\int_{\Sigma} H)^2}{|\Sigma|}.
	\end{align}
	Then substituting the above inequality \eqref{s3:key-estimate} into \eqref{s2:n=2-case} and \eqref{s2:n=3-case}, we obtain the desired inequalities \eqref{1.5} and \eqref{1.7} respectively. The equality characterization follows from the equality case of \eqref{s3:key-estimate} in Corollary \ref{cor-convex-hypersurface}. This completes the proof of Corollary \ref{cor-low-dim}.
\end{proof}

Next, we complete the proof of Theorem \ref{thm-II}.
\begin{proof}[Proof of Theorem \ref{thm-II}]
	The inequality \eqref{s1:1.1} in Theorem \ref{thm-II} has been proved. In view of \eqref{s3:Reilly-formula}, the equality in \eqref{s3:key-inequality} implies that
	\begin{align}\label{s3:Equality-I-Va}
	(\D V_a g-\nabla^2 V_a+V_a \operatorname{Ric}^{\Sigma})\(\nabla f-\frac{\nabla V_a}{V_a}f,\nabla f-\frac{\nabla V_a}{V_a}f\)=0.
	\end{align}
	On the other hand, the equality in \eqref{s3:key-step} implies that there exists a constant $\b\in\mathbb R$ such that 
	\begin{align}\label{s3:Equality-II-Va}
	(\mathring{T}_k)_{ij}=\b \mathring{A}_{ij}=\b \(f_{ij}-\frac{(V_a)_{ij}}{V_a}f-\frac{1}{n}\(\D f-\frac{\D V_a}{V_a}f\)g_{ij}\).
	\end{align}
	By assumption that $\D V_a g-\nabla^2 V_a+V_a \operatorname{Ric}^{\Sigma}>0$ on $\Sigma$, by \eqref{s3:Equality-I-Va} we have $f=\a V_a$ for some $\a\in \mathbb R$ on $\Sigma$ and hence $\-H_k^{V_a}=H_k$ on $\Sigma$ in view of Eq. \eqref{Neumann-eq}. The remaining proof is almost the same as that in the proof of Theorem \ref{thm-I}. The proof of Theorem \ref{thm-II} is completed. 
\end{proof}

\begin{rem} The sub-static condition of $(\Sigma,g,V_a)$ can be expressed in terms of certain convexity of the hypersurface. Since $\-\nabla^2 V_a=-K V_a \-g$, it follows from the Gauss equations and Gauss formula that
	\begin{align*}
	\D V_a=&\-\D V_a-\-\nabla^2 V_{a}(\nu,\nu)-H (V_a)_{\nu}=-nKV_a- H (V_a)_{\nu},\\
	(V_a)_{ij} =&(V_a)_{,ij}-(V_a)_{\nu} h_{ij}=-K V_a g_{ij}-(V_a)_{\nu} h_{ij}, \\
		\Ric^{\Sigma}_{ij}=&Hh_{ij}-h_{ik}h^k_j+(n-1)K g_{ij},
		\end{align*} 
	where we use $(V_a)_{ij}=\nabla_i\nabla_j V_a$ and $(V_a)_{,ij}=\-\nabla_i\-\nabla_j V_a$. Then we have
		\begin{align}\label{s4:sub-static-expression}
		\D V_a g_{ij}-(V_a)_{ij} +V_a \Ric^{\Sigma}_{ij}=(V_a h_{ik}-(V_a)_{\nu} g_{ik})(H g_{kj}-h_{kj}).
		\end{align}
		So if $\Sigma$ is convex and $h_{ij}\geq \frac{(V_a)_{\nu}}{V_a}g_{ij}$, then $(\Sigma,g,V_a)$ is sub-static.
\end{rem} 

\begin{proof}[Proof of Corollary \ref{cor-weighted-free-boundary-Euclidean}]
	Since $\Sigma$ is a strictly convex free boundary hypersurface in a Euclidean ball, there exists a constant unit vector $a\in \mathbb R^{n+1}$ such that $V_a=\langle x,a\rangle>0$ and $(V_a)_{\nu}=\langle a,\nu\rangle<0$, see \cite[Lems. 11 \& 12]{Lambert-Scheuer2016}. Then $(\Sigma,g,V_a)$ is strictly sub-static in view of \eqref{s4:sub-static-expression}. Applying Theorem \ref{thm-II}, we complete the proof of Corollary \ref{cor-weighted-free-boundary-Euclidean}.
\end{proof}

\end{document}